\documentclass{amsart}
\usepackage{amsfonts,amscd,amsthm,amsgen,amsmath,amssymb}
\usepackage[all]{xy}
\usepackage{epsfig,color}
\usepackage[vcentermath]{youngtab}
\newtheorem{theorem}{Theorem}[section]
\newtheorem{lemma}[theorem]{Lemma}
\newtheorem{corollary}[theorem]{Corollary}

\theoremstyle{definition}
\newtheorem{definition}[theorem]{Definition}

\theoremstyle{remark}
\newtheorem{remark}[theorem]{Remark}

\numberwithin{equation}{section}

\begin{document}

\title[An adjunction inequality obstruction to isotopy]{An adjunction inequality obstruction to isotopy of embedded surfaces in 4-manifolds}
\author{David Baraglia}

\address{School of Mathematical Sciences, The University of Adelaide, Adelaide SA 5005, Australia}

\email{david.baraglia@adelaide.edu.au}

%\subjclass[2010]{Primary 53C08, 19L50; Secondary 53D18, 53C80}

\date{\today}

%%%%%%%%%%%%%%%%%%%%%%%%%%%%%%%%%%%%%%%%%%%%%%%%%%%%%%%%%%%%%%%%%%%%%%%%%%%%%%%%
%%%%%%%%%%%%%%%%%%%%%%%%%%%%%%%%%%%%%%%%%%%%%%%%%%%%%%%%%%%%%%%%%%%%%%%%%%%%%%%%
\begin{abstract}
Consider a smooth $4$-manifold $X$ and a diffeomorphism $f : X \to X$. We give an obstruction in the form of an adjunction inequality for an embedded surface in $X$ to be isotopic to its image under $f$. It follows that the minimal genus of a surface representing a given homology class and which is isotopic to its image under $f$ is generally larger than the minimal genus without the isotopy condition. We give examples where the inequality is strict. We use our obstruction to construct examples of infinitely many embedded surfaces which are all continuously isotopic but mutually non-isotopic smoothly. 

\end{abstract}
%%%%%%%%%%%%%%%%%%%%%%%%%%%%%%%%%%%%%%%%%%%%%%%%%%%%%%%%%%%%%%%%%%%%%%%%%%%%%%%%
%%%%%%%%%%%%%%%%%%%%%%%%%%%%%%%%%%%%%%%%%%%%%%%%%%%%%%%%%%%%%%%%%%%%%%%%%%%%%%%%

\maketitle

%%%%%%%%%%%%%%%%%%%%%%%%%%%%%%%%%%%%%%%%%%%%%%%%%%%%%%%%%%%%%%%%%%%%%%%%%%%%%%%%
%%%%%%%%%%%%%%%%%%%%%%%%%%%%%%%%%%%%%%%%%%%%%%%%%%%%%%%%%%%%%%%%%%%%%%%%%%%%%%%%
%%%%%%%%%%%%%%%%%%%%%%%%%%%%%%%%%%%%%%%%%%%%%%%%%%%%%%%%%%%%%%%%%%%%%%%%%%%%%%%%
%%%%%%%%%%%%%%%%%%%%%%%%%%%%%%%%%%%%%%%%%%%%%%%%%%%%%%%%%%%%%%%%%%%%%%%%%%%%%%%%

\section{Introduction}

The minimal genus problem for smooth $4$-manifolds asks to find the minimal genus of an embedded oriented surface representing a given $2$-dimensional homology class in $4$-manifold. Seiberg-Witten theory has generated vast progress on the minimal genus problem \cite{km,mst,os}, see also the survey paper \cite{law}. In particular one has the adjunction inequality: {\em let $X$ be a compact, oriented, smooth $4$-manifold with $b^+(X)>1$. Suppose $\Sigma \to X$ is an embedded oriented surface of genus $g$ and non-negative self-intersection.
\begin{itemize}
\item{If $g \ge 1$ then
\[
2g-2 \ge [\Sigma]^2 + | \langle [\Sigma] , c \rangle |
\]
for every basic class $c \in H^2( X ; \mathbb{Z})$ (recall that $c$ is {\em basic} if $c = c_1(\mathfrak{s})$ for some spin$^c$-structure $\mathfrak{s}$ with non-zero Seiberg-Witten invariant).}
\item{If $g=0$ and a basic class exists, then $[\Sigma]$ is a torsion class.}
\end{itemize}}

In this paper we consider an extension of the minimal genus problem where the embedded surface is required to satisfy an additional constraint. Using a parametrised version of Seiberg-Witten theory, we obtain adjunction inequalities giving a lower bound on the genus of such surfaces. These adjunction inequalities can be non-trivial even on $4$-manifolds with trivial Seiberg-Witten invariants.

Let $X$ be a compact, oriented, smooth $4$-manifold and $f : X \to X$ an orientation preserving diffeomorphism. Assume further that $b^+(X) \ge 3$. We are interested in studying embedded surfaces $i :\Sigma \to X$ such that $\Sigma$ is smoothly isotopic to $f(\Sigma)$. That means there is a smooth family of embeddings $\{ i_t : \Sigma \to X\}_{0 \le t \le 1}$ such that $i_0 = i$ and $i_1 = f \circ i$. Our adjunction inequality will be expressed in terms of certain classes in $H^2(X ; \mathbb{Z})$ which we call {\em $f$-basic}. Let $\mathfrak{s}$ be a spin$^c$-structure on $X$ which is preserved by $f$ and satisfying $d(X,\mathfrak{s}) = -1$, where 
\[
d(X , \mathfrak{s}) = \frac{ c_1(\mathfrak{s})^2 - \sigma(X)}{4} - 1 + b_1(X) - b^+(X)
\]
is the virtual dimension of the Seiberg-Witten moduli space associated to $(X,\mathfrak{s})$. In this case one can associate to $f$ a Seiberg-Witten invariant $SW^{\mathbb{Z}_2}(X , f , \mathfrak{s}) \in \mathbb{Z}_2$ (\cite{rub1}, \cite{bk}). Roughly, $SW^{\mathbb{Z}_2}(X , f , \mathfrak{s})$ is a mod $2$ count of the number of solutions of the Seiberg-Witten equations for the $1$-parameter family of $4$-manifolds given by the mapping torus of $f$. If $f$ satisfies an additional orientability condition (see \textsection \ref{sec:swdiff}) then we can count solutions with sign, giving an integer valued invariant $SW^{\mathbb{Z}}(X , f , \mathfrak{s}) \in \mathbb{Z}$ which reduces mod $2$ to $SW^{\mathbb{Z}_2}(X , f , \mathfrak{s})$. We refer to $SW^{\mathbb{Z}_2}(X , f , \mathfrak{s})$ and $SW^{\mathbb{Z}}(X , f , \mathfrak{s})$ (when defined) as the Seiberg-Witten invariants of $f$.

\begin{definition}\label{def:basic}
We say that a cohomology class $c \in H^2(X ; \mathbb{Z})$ is {\em $f$-basic} if $c = c_1(\mathfrak{s})$ for a spin$^c$-structure $\mathfrak{s}$ on $X$ which is preserved by $f$, satisfies $d(X , \mathfrak{s}) = -1$ and either $SW^{\mathbb{Z}_2}(X , f , \mathfrak{s})$ or $SW^{\mathbb{Z}}(X , f , \mathfrak{s})$ (if defined) is non-zero.
\end{definition}

Our adjunction inequality for embedded surfaces isotopic to their image under $f$ is as follows:

\begin{theorem}\label{thm:fadjunction}
Let $X$ be a compact, oriented smooth $4$-manifold with $b^+(X) \ge 3$ and let $f : X \to X$ be an orientation preserving diffeomorphism. Let $\Sigma \to X$ be an embedded oriented surface of genus $g$ with $[\Sigma]^2 \ge 0$ and suppose that $\Sigma$ is smoothly isotopic to $f(\Sigma)$.
\begin{itemize}
\item{If $g \ge 1$, then $\Sigma$ satisfies the adjunction inequality
\[
2g-2 \ge [\Sigma]^2 + | \langle [\Sigma] , c \rangle |.
\]
for every $f$-basic class $c$.
}
\item{If $g=0$ and an $f$-basic class exists, then $[\Sigma]$ is a torsion class.}
\end{itemize}

\end{theorem}

Note Theorem \ref{thm:fadjunction} can be used as a vanishing theorem for the Seiberg-Witten invariants of diffeomorphisms. Namely if there exists an embedded surface $\Sigma$ which is isotopic to its image under $f$ and which violates the adjunction inequality for some characteristic $c$, then $c$ is not a basic class.

To extract interesting results from Theorem \ref{thm:fadjunction}, we need to know examples of diffeomorphisms whose associated Seiberg-Witten invariants are non-trivial. A large class of examples can be obtained by a connected sum construction \cite{rub3}, \cite{bk}. As an application of Theorem \ref{thm:fadjunction} combined with the connected sum construction, we obtain examples of infinite classes of embedded surfaces which are all continuously isotopic but mutually non-isotopic smoothly. Before stating the result we comment on some related works. The paper \cite{fs} constructs embedded surfaces which are mapped to each other by an ambient homeomorphism but can not be mapped to each other by an ambient diffeomorphism. The paper \cite{ak} constructs examples of non-isotopic $2$-spheres which become isotopic on connected sum with $S^2 \times S^2$ and \cite{akmr} constructs infinite families of topologically isotopic but smoothly non-isotopic embedded $2$-spheres which become isotopic on connected sum with $S^2 \times S^2$.

Recall that a class $u \in H^2(X ; \mathbb{Z})$ is called {\em ordinary} if it not characteristic.

\begin{theorem}\label{thm:isotopysurface}
Let $X$ be one of the $4$-manifolds 
\begin{itemize}
\item[(i)]{$\#^{n} (S^2 \times S^2) \#^n K3$ for $n \ge 2$, or}
\item[(ii)]{$\#^{2n} \mathbb{CP}^2 \#^{m} \overline{\mathbb{CP}^2}$ for $n \ge 2$, $m \ge 10n+1$.}
\end{itemize}
Then there exists a diffeomorphism $f : X \to X$ such that:
\begin{itemize}
\item[(1)]{$f$ is homotopic to the identity.}
\item[(2)]{Let $u \in H^2(X ; \mathbb{Z})$ be any class with $u^2 > 0$. If $\Sigma \to X$ is an embedded surface of genus $g$ representing $u$ and $f^n(\Sigma)$, $f^m(\Sigma)$ are smoothly isotopic for some $n \neq m$, then $2g-2 \ge u^2$.}
\item[(3)]{Let $u \in H^2(X ; \mathbb{Z})$ satisfy $u^2 > 0$. In case (ii) assume moreover that u is a multiple of a primitive ordinary class. Then there exists an embedded surface $\Sigma \to X$ representing $u$ such that the surfaces $\{ f^n(\Sigma) \}_{n \in \mathbb{Z} }$ are all continuously isotopic but mutually non-isotopic smoothly. If $u$ is primitive then $\Sigma$ can be taken to have genus zero. In all other cases we can take $\Sigma$ to have genus $g$ satisfying $2g-2 < u^2$.}
\end{itemize}
\end{theorem}

\begin{remark}
In Theorem \ref{thm:isotopysurface} (3), if $u$ is primitive then it is in addition possible to choose the embedding $\Sigma \to X$ to have simply-connected complement. It follows that the smoothly non-isotopic surfaces $\{ f^n(\Sigma) \}$ are all smoothly isotopic after taking connected sum with a single copy of $S^2 \times S^2$, according to the main theorem of \cite{akmrs} (which is proven using the $4$-dimensional light bulb theorem of Gabai \cite{gab}). To be more precise, for any $n,m$, the surfaces $f^n(\Sigma), f^m(\Sigma)$ become isotopic after taking the connected sum of $X$ with $S^2 \times S^2$, summing at a point disjoint from $f^n(\Sigma) \cup f^m(\Sigma)$.
\end{remark}

It is interesting to contrast Theorem \ref{thm:isotopysurface} with the case of knots in $S^3$. Cerf's theorem \cite{cer} implies that if two smooth knots in $S^3$ are mapped onto one another by an orientation preserving diffeomorphism, then the two knots are ambiently isotopic. An analogous problem for embedded surfaces in smooth $4$-manifolds was posed in \cite{akmr}:\\

\noindent {\bf Problem 1:} Suppose two embedded surfaces in a smooth $4$-manifold are mapped onto one another by a diffeomorphism homotopic to the identity. Then are the surfaces smoothly isotopic?\\

The authors of \cite{akmr} suspected that Problem 1 does not have a positive solution in general. Here we confirm that this is indeed the case.

\begin{theorem}
The analogue of Cerf's theorem for embedded surfaces in $4$-manifolds (Problem 1) fails in general: there exist compact smooth $4$-manifolds $X$, embedded surfaces $i_1, i_2 : \Sigma \to X$ and a diffeomorphism $f : X \to X$ homotopic to the identity for which $i_2 = f \circ i_1$, but $i_1, i_2$ are not smoothly isotopic.
\end{theorem}
\begin{proof}
Let $X$ and $f : X \to X$ be as in the statement of Theorem \ref{thm:isotopysurface}. Then according to this theorem, there exists embedded surfaces $\Sigma \to X$ such that $\Sigma$ and $f(\Sigma)$ are not smoothly isotopic. On the other hand the diffeomorphism $f$, which is homotopic to the identity, maps $\Sigma$ onto $f(\Sigma)$.
\end{proof}

We have just shown that if $X$ is one of the $4$-manifolds in the statement of Theorem \ref{thm:isotopysurface}, then Problem 1 fails in a very strong sense. An interesting problem would be to find examples of compact, smooth $4$-manifolds where Problem 1 has a positive solution, if there are any.

A brief outline of the contents of this paper is as follows. In Section \ref{sec:swdiff} we recall the definition of the Seiberg-Witten invariants of a diffeomorphism and recall how the gluing procedure of \cite{bk} can be used to produce examples of diffeomorphisms for which this invariant is non-zero. In Section \ref{sec:adj} we prove the main result of this paper, Theorem \ref{thm:fadjunction}. In Section \ref{sec:noniso} we give our main application of Theorem \ref{thm:fadjunction}, which is the construction of infinite families of embedded surfaces which are continuously isotopic but mutually non-isotopic smoothly.\\

\noindent{\bf Acknowledgments.} The author was financially supported by the Australian Research Council Discovery Project DP170101054.

%%%%%%%%%%%%%%%%%%%%%%%%%%%%%%%%%%%
\section{Seiberg-Witten invariants of diffeomorphisms}\label{sec:swdiff}

Let $X$ be a compact, oriented, smooth $4$-manifold and $f : X \to X$ an orientation preserving diffeomorphism. Assume also that $b^+(X) \ge 3$. In this section we recall how one can define a notion of Seiberg-Witten invariants for $f$, depending only on the smooth isotopy class of $f$.

Let $E_f \to S^1$ denote the mapping torus of $f$, which we think of as being a smooth family of $4$-manifolds parametrised by the circle. Let $\mathfrak{s}$ be a spin$^c$-structure on $X$ which is preserved by $f$ and satisfies $d(X,\mathfrak{s}) = -1$, where 
\[
d(X , \mathfrak{s}) = \frac{ c_1(\mathfrak{s})^2 - \sigma(X)}{4} - 1 + b_1(X) - b^+(X)
\]
is the virtual dimension of the Seiberg-Witten moduli space associated to $(X,\mathfrak{s})$. Then as described in \cite{bk}, after fixing a choice of a smoothly varying family of metrics and self-dual $2$-form perturbations, we can construct a families Seiberg-Witten moduli space $\mathcal{M}(E_f , \mathfrak{s})$ parametrising solutions to the Seiberg-Witten equations on the fibres of $E_f$. For sufficiently generic perturbations, the moduli space $\mathcal{M}(E_f , \mathfrak{s})$ is a compact manifold of dimension $d(X , \mathfrak{s}) + 1 = 0$. By the usual cobordism argument, the number of points of $\mathcal{M}(E_f , \mathfrak{s})$ counted mod $2$ is independent of the choice of families metric and perturbation and hence defines an invariant
\[
SW^{\mathbb{Z}_2}( X , f , \mathfrak{s}) \in \mathbb{Z}_2,
\]
which we shall refer to as the Seiberg-Witten invariant of $(X,f , \mathfrak{s})$ (or less precisely as the Seiberg-Witten invariants of $f$).

The mod $2$ Seiberg-Witten invariants of $f$ can be lifted to integer invariants provided $f$ satisfies an orientability condition which we now describe. The Grassmannian $G_+$ of oriented maximal positive definite subspaces of $H^2(X ; \mathbb{R})$ has two connected components corresponding to the two possible orientations one can put on any particular maximal positive definite subspace. The diffeomorphism $f : X \to X$ is orientation preserving, so induces an isometry of $H^2(X ; \mathbb{R})$, which in turn induces a continuous map of $G_+$ to itself. Let us define $sgn_+(f) = \pm 1$ to be $+1$ if $f$ sends each component of $G_+$ to itself and $sgn_+(f) = -1$ if $f$ exchanges the two components of $G_+$.

Suppose that $sgn_+(f)=1$. In this case it follows from \cite{bk} that the families moduli space $\mathcal{M}(E_f , \mathfrak{s})$ is orientable. Moreover the two possible orientations of $\mathcal{M}(E_f , \mathfrak{s})$ correspond in a natural way to the two connected component of $G_+$. From this it can be seen that if the metric and perturbation are varied, then the cobordism between the two moduli spaces can be promoted to an oriented cobordism. Therefore if $sgn_+(f)=1$ then (after choosing a connected component of $G_+$) we obtain an integer invariant
\[
SW^{\mathbb{Z}}(X , f , \mathfrak{s}) \in \mathbb{Z}
\]
by counting with signs the number of points of $\mathcal{M}(E_f , \mathfrak{s})$.

\begin{remark}
The Seiberg-Witten invariants of $f$, when defined, are invariants of the smooth isotopy class of $f$. One sees this by noting that the underlying family of $4$-manifolds $E_f$ given by the mapping torus of $f$ is determined up to isomorphism by the isotopy class of $f$.

In the case $f = id$ one easily sees that the Seiberg-Witten invariants of $f$ are zero because the family $E_f$ is just the constant family over $S^1$ and the virtual dimension $d(X , \mathfrak{s}) = -1$ of the unparametrised Seiberg-Witten moduli space is negative. Therefore if an orientation preserving diffeomorphism $f$ has non-zero Seiberg-Witten invariants, it follows that $f$ is not smoothly isotopic to the identity. This observation was used in \cite{rub3} and \cite{bk} to construct examples of diffeomorphisms which are continously isotopic to the identity but not smoothly isotopic. Although it is not the main focus of this paper, we will construct new examples of this phenomenon below.
\end{remark}

Let $M$ be a smooth compact simply-connected $4$-manifold. We will say that $M$ {\em dissolves on connected sum with $S^2 \times S^2$} if the connected sum $M \# (S^2 \times S^2)$ is diffeomorphic to
\[
\#^n \mathbb{CP}^2 \#^m \overline{\mathbb{CP}^2} \; \; \text{or} \; \; \#^n (S^2 \times S^2) \#^m K3
\]
for some $n,m \ge 0$, where $K3$ denotes the underlying $4$-manifold of a $K3$ surface. All examples that we consider will have non-positive signature, so the diffeomorphism from $M \# (S^2 \times S^2)$ to $\#^n \mathbb{CP}^2 \#^m \overline{\mathbb{CP}^2} \; \; \text{or} \; \; \#^n (S^2 \times S^2) \#^m K3$ can taken to be orientation preserving.

The following result is a special case of \cite[Theorem 9.7]{bk} and is proven using a gluing formula for the families Seiberg-Witten invariant:
\begin{theorem}\label{thm:dissolve1}
Let $M$ be a compact simply-connected smooth $4$-manifold with $b^+(M)>1$ and $\sigma(M) \le 0$. Assume also that $M$ is not homeomorphic to $K3$. Suppose that $M$ dissolves on connected sum with $S^2 \times S^2$ and that there exists a spin$^c$-structure $\mathfrak{s}$ on $M$ with $d(M , \mathfrak{s}) = 0$ and $SW(M , \mathfrak{s}) = 1 \; ({\rm mod} \; 2)$. Let $X = M \# (S^2 \times S^2)$. Then there exists a diffeomorphism $f : X \to X$ such that $f$ is continuously isotopic to the identity and a spin$^c$-structure $\mathfrak{s}_X$ on $X$ such that $d(X , \mathfrak{s}_X) = -1$ and $SW^{\mathbb{Z}}(X , f , \mathfrak{s}_X) \neq 0$.
\end{theorem}
\begin{proof}
Since $M$ dissolves on connected sum with $S^2 \times S^2$ and $\sigma(M)\le 0$, we have that $X = M \# (S^2 \times S^2)$ is diffeomorphic to either $\#^a \mathbb{CP}^2 \#^b \overline{\mathbb{CP}^2}$ for some $a,b$, if $M$ is not spin or $\#^n (S^2 \times S^2) \#^m K3$ for some $n,m$ if $M$ is spin. Clearly $a,b > 2$ in the non-spin case and $n \ge 1$ in the spin case. Let $M'$ be given by $\#^{a-1} \mathbb{CP}^2 \#^{b-1} \overline{\mathbb{CP}^2}$ if $M$ is not spin and $M' = \#^{n-1} (S^2 \times S^2) \#^m K3$ if $M$ is spin. Then $M,M'$ are homeomorphic and become diffeomorphic on connected sum with $S^2 \times S^2$. Moreover, all of the Seiberg-Witten invariants of $M'$ are zero since in all cases $M'$ is a connected sum of two smooth $4$-manifolds which are not negative definite (note that we explicitly required $M$ not to be homeomorphic to $K3$ in order to rule out the case $n=0,m=1$ which would give $M' = K3$).

Now $M,M'$ satisfy the conditions of \cite[Theorem 9.7]{bk}. According to the proof of \cite[Theorem 9.7]{bk}, there exist diffeomorphisms $f_1,f_2 : X \to X$ inducing the same action on $H^2(X ; \mathbb{Z})$ and a spin$^c$-structure $\mathfrak{s}_X$ with $d(X , \mathfrak{s}_X) = -1$ preserved by $f_1$ and $f_2$ and such that $SW^{\mathbb{Z}_2}(X , f_1 , \mathfrak{s}_X) \neq SW^{\mathbb{Z}_2}(X , f_2 , \mathfrak{s}_X)$. Let $f = f_1 \circ f_2^{-1}$. Then $f$ acts trivially on $H^2(X ; \mathbb{Z})$, which by \cite{qui} implies that $f$ is continuously isotopic to the identity. By \cite[Lemma 2.6]{rub1}, we have
\[
SW^{\mathbb{Z}_2}(X , f_1 \circ f_2^{-1} , \mathfrak{s}_X) = SW^{\mathbb{Z}_2}(X , f_1 , \mathfrak{s}_X) - SW^{\mathbb{Z}_2}(X , f_2 , \mathfrak{s}_X)
\]
and hence $SW^{\mathbb{Z}_2}(X , f , \mathfrak{s}_X) \neq 0$. Of course this implies that $SW^{\mathbb{Z}}(X , f , \mathfrak{s}_X)$ is also non-zero. Note also that $SW^{\mathbb{Z}}(X , f , \mathfrak{s}_X)$ is defined because $sgn_+(f) = 1$, as $f$ acts trivially on $H^2(X ; \mathbb{Z})$.
\end{proof}

\begin{corollary}\label{cor:ctsnotsmooth}
Let $X$ be one of the $4$-manifolds 
\begin{itemize}
\item[(i)]{$\#^{n} (S^2 \times S^2) \#^n K3$ for any $n \ge 2$, or}
\item[(ii)]{$\#^{2n} \mathbb{CP}^2 \#^{m} \overline{\mathbb{CP}^2}$ for $n \ge 2$, $m \ge 10n+1$.}
\end{itemize}
Then there exists a diffeomorphism $f : X \to X$ such that $f$ is continuously isotopic to the identity and a spin$^c$-structure $\mathfrak{s}$ on $X$ such that $d(X , \mathfrak{s}) = -1$ and $SW^{\mathbb{Z}}(X , f , \mathfrak{s}) \neq 0$, hence $f$ is not smoothly isotopic to the identity.
\end{corollary}
\begin{proof}
This result was proven in \cite[Corollary 9.8]{bk}. We briefly recall the proof. First recall from \cite{gom} that the elliptic surfaces $E(n)$ dissolve on connected sum with $S^2 \times S^2$. For case (i) we take $M = E(2n)$ which has even intersection form. It follows that $X = M \# (S^2 \times S^2) \cong \#^{n} (S^2 \times S^2) \#^n K3$. Then the result follows from Theorem \ref{thm:dissolve1} and the fact that on a symplectic $4$-manifold with $b^+(M)>1$ (such as $M$) there exists a spin$^c$-structure $\mathfrak{s}$ with $d(M , \mathfrak{s}) = 0$ and $SW(M , \mathfrak{s}) = 1 \; ({\rm mod} \; 2)$, namely the canonical spin$^c$-structure associated to a compatible almost complex structure.

In case (ii), we take $M = E(n) \#^{m-10n} \overline{\mathbb{CP}^2}$, which has odd intersection form since $m - 10n \ge 1$. Then $X = M \# (S^2 \times S^2) \cong \#^{2n} \mathbb{CP}^2 \#^{m} \overline{\mathbb{CP}^2}$. Moreover, since $M$ is the blowup of a symplectic $4$-manifold with $b^+(M)>1$, there exists a spin$^c$-structure $\mathfrak{s}$ with $d(M , \mathfrak{s}) = 0$ and $SW(M , \mathfrak{s}) = 1 \; ({\rm mod} \; 2)$ by the blowup formula. So we can again apply Theorem \ref{thm:dissolve1}.
\end{proof}

\begin{remark}
Let $f : X \to X$ be as in Corollary \ref{cor:ctsnotsmooth}. For any integer $n$, let $f^n$ denote the $n$-fold composition of $f$ with itself. Suppose that $SW^{\mathbb{Z}}(X , f , \mathfrak{s}) = k \neq 0$. Then it follows from \cite[Lemma 2.6]{rub1} that $SW^{\mathbb{Z}}(X , f^n , \mathfrak{s}) = nk$, hence the diffeomorphisms $\{ f^n \}_{n \in \mathbb{Z}}$ are all continuously isotopic to the identity but mutually non-isotopic smoothly.
\end{remark}

%%%%%%%%%%%%%%%%%%%%%%%%%%%%%%%%%%%%%%%
\section{The adjunction inequality}\label{sec:adj}

In this section we will prove Theorem \ref{thm:fadjunction}. Before getting to the proof we need a few preliminary results. Let $\Sigma$ denote a closed oriented surface equipped with a Riemannian metric $g_\Sigma$ and let $\mu$ be a $1$-form on $\Sigma$. For any $t \in [0,1]$ we define a metric $g'$ on $\mathbb{R} \times \Sigma$ by
\[
g' = dx^2 + t(  dx \otimes \mu + \mu \otimes dx + \mu^2) + g_\Sigma,
\]
where $x$ denotes the standard coordinate on $\mathbb{R}$.

\begin{lemma}\label{lem:scalcurv}
Let $s$ denote the scalar curvature of $g_\Sigma$ and $s'$ the scalar curvature of $g'$. Suppose that the $1$-form $\mu$ is harmonic with respect to $g_\Sigma$. Then
\[
s' = s +2\gamma \, Ric( \mu , \mu ) +\gamma | \nabla(\mu) |^2 + 2 \gamma^2 | \nabla_{\mu^{\#}}(\mu) |^2
\]
where $Ric$ and $\nabla$ denote the Ricci curvature and Levi-Civita connection of $g_\Sigma$, $\mu^{\#} = (g_\Sigma)^{-1}(\mu)$, $\gamma$ is given by
\[
\gamma = \frac{ t^2-t}{1+ (t-t^2)|\mu|^2}
\]
and where $| \; . \; |$ denotes taking the norm of various tensors with respect to $g_\Sigma$.
\end{lemma}
\begin{proof}
The proof is a direct computation. We use local coordinates $(x^1,x^2)$ on $\Sigma$. Setting $x^0 = x$, we have local coordinates $(x^0,x^1,x^2)$ on $\mathbb{R} \times \Sigma$. We will use index notation along with the summation convention. Indices $a,b,c, \dots $ will run from $0$ to $2$ while indices $i,j,k,\dots$ will run from $1$ to $2$ only. For instance we may write
\[
g_\Sigma = g_{ij}dx^i dx^j \text{ and } g' = g'_{ab} dx^a dx^b.
\]
Then $g'$ is given in components by
\[
g'_{00} = 1, \quad g'_{0i} = t\mu_i, \quad g'_{ij} = g_{ij} + t\mu_i \mu_j.
\]
One then finds
\[
{g'}^{00} = \alpha, \quad {g'}^{0i} = \beta \mu^i, \quad {g'}^{ij} = g^{ij} + \gamma \mu^i \mu^j,
\]
where $\mu^i = g^{ij}\mu_j$ and
\[
\alpha = \frac{1+t|\mu|^2}{1+(t-t^2)|\mu|^2}, \quad \beta = \frac{-t}{1+(t-t^2)|\mu|^2}, \quad \gamma = \frac{t^2-t}{1+(t-t^2)|\mu|^2}.
\]
Next, we compute the Christoffel symbols of $g'$ in terms of those of $g_\Sigma$. Let $\Gamma_{kij}$ and $\Gamma'_{cab}$ be given by
\[
2\Gamma_{kij} = \partial_i g_{jk} + \partial_j g_{ik} - \partial_k g_{ij}, \quad 2\Gamma'_{cab} = \partial_a g'_{bc} + \partial_b g'_{ac} - \partial_c g'_{ab}.
\]
Let $\mu_{ij} = \partial_i \mu_j$. Note that $\mu$ is a closed $1$-form since it is harmonic and thus $\mu_{ij} = \mu_{ji}$. A short calculation gives
\[
\Gamma'_{kij} = \Gamma_{kij} + t\mu_{ij}\mu_k, \quad \Gamma'_{0ij} = t\mu_{ij}
\]
and $\Gamma'_{cab}=0$ whenver $a=0$ or $b=0$. Raising the first index gives:
\begin{equation*}
\begin{aligned}
{{\Gamma'}^{0}}_{ij} &= {g'}^{00} \Gamma'_{0ij} + {g'}^{0k} \Gamma'_{kij} \\
&= \alpha t \mu_{ij} + \beta \mu^k ( \Gamma_{kij} + t \mu_{ij} \mu_k) \\
&= (\alpha t + t\beta |\mu|^2) \mu_{ij} + \beta \mu^k \Gamma_{kij} \\
&= -\beta( \mu_{ij} - {\Gamma^k}_{ij}\mu_k) \\
&= -\beta( \partial_i \mu_j - {\Gamma^k}_{ij} \mu_k ) \\
&= -\beta \nabla_i \mu_j
\end{aligned}
\end{equation*}
and
\begin{equation*}
\begin{aligned}
{{\Gamma'}^{k}}_{ij} &= {g'}^{k0} \Gamma'_{0ij} + {g'}^{kl} \Gamma'_{lij} \\
&= \beta \mu^k t \mu_{ij} + (g^{kl} + \gamma \mu^k \mu^l)( \Gamma_{lij} + t\mu_{ij} \mu_l) \\
&= \beta\mu^k t \mu_{ij} + {\Gamma^k}_{ij} + t \mu^k \mu_{ij} + \gamma \mu^k \mu^l \Gamma_{lij} + t \gamma |\mu|^2 \mu^k \mu_{ij} \\
&= {\Gamma^k}_{ij} + (t \beta + t + t \gamma |\mu|^2) \mu^k \mu_{ij} + \gamma \mu^k \mu^l \Gamma_{lij} \\
&= {\Gamma^k}_{ij} - \gamma (\mu_{ij} - {\Gamma^l}_{ij} \mu_l) \mu^k \\
&= {\Gamma^k}_{ij} - \gamma (\nabla_i \mu_j) \mu^k.
\end{aligned}
\end{equation*}
Therefore, if $\nabla$ denotes the Levi-Civita connection of $g_\Sigma$ and $\nabla'$ the Levi-Civita connection of $g'$, we have shown that
\[
\nabla'_0( \partial_i) = \nabla'_i( \partial_0) = \nabla'_0(\partial_0) = 0
\]
and
\[
\nabla'_i(\partial_j) = \nabla_i(\partial_j) + \nabla_i(\mu_j) \theta
\]
where
\[
\theta = -\beta \partial_0 - \gamma \mu^{\#}, \quad \mu^{\#} = \mu^j \partial_j.
\]

This can alternatively be written as
\[
\nabla'_i ( \partial_j ) = \nabla_i \partial_j + i_{\partial_j} (\nabla_i \mu ) \theta
\]
where $i_{\partial_j}$ denotes contraction by $\partial_j$.

Let ${R_{ijk}}^l$ and ${R'_{ijk}}^{l}$ be the Riemann curvatures of $g_\Sigma$ and $g'$, so
\[
\nabla_i \nabla_j \partial_k - \nabla_j \nabla_i \partial_k = {R_{ijk}}^{l} \partial_l, \quad \nabla'_a \nabla'_b \partial_c - \nabla'_b \nabla'_a \partial_c = {R'_{abc}}^{d} \partial_d.
\]
It follows that ${R'_{abc}}^{d} = 0$ whenever $a,b$ or $c$ are zero, so the only curvature terms for $g'$ are of the form ${R'_{ijk}}^{0}$ or ${R'_{ijk}}^{l}$. Moreover only the ${R'_{ijk}}^{l}$ terms contribute to the Ricci curvature, so we only need to compute these. We have:

\begin{align*}
\nabla'_i \nabla'_j \partial_k &= \nabla'_i ( \nabla_j \partial_k + (\nabla_j \mu_k) \theta ) \\
&= \nabla_i \nabla_j \partial_k + \iota_{ (\nabla_j \partial_k )} (\nabla_i \mu ) \theta + \partial_i( \nabla_j \mu_k ) \theta + (\nabla_j \mu_k) \nabla'_i (\theta).
\end{align*}
We also have
\begin{align*}
\nabla'_i(\theta) &= \nabla'_i( -\beta \partial_0 - \gamma \mu^{\#} ) = - \partial_i (\beta) \partial_0 - \gamma_i \mu^{\#} - \gamma \nabla_i ( \mu^{\#} ) - \gamma \iota_{\mu^{\#}}(\nabla_i \mu) \theta \\
&= \nabla'_i( -\beta \partial_0 - \gamma \mu^{\#} ) = - \partial_i (\beta) \partial_0 - \gamma_i \mu^{\#} - \gamma \nabla_i ( \mu^{\#} ) - \gamma \mu^m (\nabla_i \mu_m) \theta
\end{align*}
where $\gamma_i = \partial_i \gamma$. Therefore
\begin{align*}
{R'_{ijk}}^{l} &= {R_{ijk}}^{l} - \iota_{ (\nabla_j \partial_k )} (\nabla_i \mu ) \gamma \mu^l - \partial_i( \nabla_j \mu_k ) \gamma \mu^l - (\nabla_j \mu_k)(\gamma_i \mu^l + \gamma \nabla_i\mu^l) + \gamma^2 (\nabla_j \mu_k) \mu^m (\nabla_i \mu_m) \mu^l \\
& \quad \quad  +\iota_{ (\nabla_i \partial_k )} (\nabla_j \mu ) \gamma \mu^l + \partial_j( \nabla_i \mu_k ) \gamma \mu^l + (\nabla_i \mu_k)(\gamma_j \mu^l + \gamma \nabla_j\mu^l) -  \gamma^2 (\nabla_i \mu_k) \mu^m (\nabla_j \mu_m) \mu^l \\
&= {R_{ijk}}^{l} - {\Gamma^m}_{jk} (\nabla_i \mu_m ) \gamma \mu^l - \partial_i( \nabla_j \mu_k ) \gamma \mu^l - (\nabla_j \mu_k)(\gamma_i \mu^l + \gamma \nabla_i\mu^l) + \gamma^2 (\nabla_j \mu_k) \mu^m (\nabla_i \mu_m) \mu^l \\
& \quad \quad  + {\Gamma^m}_{ik} (\nabla_j \mu_m ) \gamma \mu^l + \partial_j( \nabla_i \mu_k ) \gamma \mu^l + (\nabla_i \mu_k)(\gamma_j \mu^l + \gamma \nabla_j\mu^l) - \gamma^2 (\nabla_i \mu_k) \mu^m (\nabla_j \mu_m) \mu^l  \\
& = {R_{ijk}}^{l} + ((\nabla_j \nabla_i \mu_k) - (\nabla_i \nabla_j \mu_k )) \gamma \mu^l - (\nabla_j \mu_k)(\gamma_i \mu^l + \gamma \nabla_i\mu^l) + \gamma^2 (\nabla_j \mu_k) \mu^m (\nabla_i \mu_m) \mu^l  \\
& \quad \quad + (\nabla_i \mu_k)(\gamma_j \mu^l + \gamma \nabla_j\mu^l) + \gamma^2 (\nabla_i \mu_k) \mu^m (\nabla_j \mu_m) \mu^l.
\end{align*}

Then since $\nabla_i \nabla_j \mu_k - \nabla_j \nabla_i \mu_k = - {R_{ijk}}^{l} \mu_l$, we have

\begin{equation*}
\begin{aligned}
{R'_{ijk}}^{l} &= {R_{ijk}}^{l} + \nabla_j \mu_k( -\gamma_i \mu^l - \gamma \nabla_i (\mu^l) + \gamma^2 (\nabla_i(\mu_m)\mu^m)\mu^l) \\
& \; \; \; \; - \nabla_i(\mu_k)( -\gamma_j \mu^l - \gamma \nabla_j(\mu^l) + \gamma^2 (\nabla_j(\mu_m)\mu^m)\mu^l ) + \gamma R_{ijkm} \mu^m \mu^l.
\end{aligned}
\end{equation*}
Let $R_{ij} = {R_{kij}}^{k}$ and $R'_{ab} = {R'_{cab}}^{c}$ be the Ricci curvatures of $g_\Sigma$ and $g'$. As noted above, the only non-zero terms in $R'_{ab}$ are given by ${R'}_{ij} = {R'_{kij}}^{k}$. A direct computation gives
\begin{equation*}
\begin{aligned}
R'_{ij} &= R_{ij} - \nabla_k \mu_j ( -\gamma_i \mu^k - \gamma \nabla_i \mu^k + \gamma^2 \nabla_i (\mu_m)\mu^m \mu^k) \\
& \; \; \; \; + \nabla_i \mu_j ( -\gamma_k \mu^k - \gamma \nabla_k \mu^k + \gamma ^2 \nabla_k (\mu_m) \mu^m \mu^k ) + \gamma R_{rijs} \mu^r \mu^s \\
& = R_{ij} + \gamma_i \mu^k \nabla_k \mu_j + \gamma \nabla_i( \mu^k )\nabla_j(\mu_k) - \gamma^2 (\nabla_i(\mu_m)\mu^m)(\nabla_j(\mu_n) \mu^n) \\
& \; \; \; \; - (\nabla_i (\mu_j))(\gamma_k \mu^k) - \gamma (\nabla_k(\mu^k))\nabla_i \mu_j + \gamma^2 (\nabla_i (\mu)_j) \nabla_k(\mu_m)\mu^m \mu^k + \gamma R_{rijs} \mu^r \mu^s.
\end{aligned}
\end{equation*}
But note that $\gamma = (t^2-t)/(1+ (t-t^2)|\mu|^2)$, so
\[
\gamma_i = \frac{-(t^2-t)}{(1 + (t-t^2)|\mu|^2)^2} (t-t^2) \partial_i( |\mu|^2 ) = \gamma^2 \partial_i( |\mu|^2)
\]
and
\[
\partial_i( |\mu|^2) = \partial_i g_\Sigma(\mu,\mu) = 2 \nabla_i(\mu_k) \mu^k.
\]
So
\[
-\gamma_i \mu^k \nabla_k(\mu_j) = -2\gamma^2 (\nabla_i(\mu_m) \mu^m)(\nabla_j(\mu_n)\mu^n)
\]
and
\[
\gamma_k \mu^k = 2 \gamma^2 \nabla_k(\mu_m)\mu^k \mu^m.
\]
Substituting these into the above expression for $R'_{ij}$, we get
\begin{equation*}
\begin{aligned}
R'_{ij} &= R_{ij} + \gamma \nabla_i( \mu^k) \nabla_j( \mu_k) + \gamma^2 (\nabla_i(\mu_m) \mu^m)(\nabla_j(\mu_n)\mu^n) \\
& \; \; \; \; - \gamma^2 (\nabla_i(\mu)_j) ( \nabla_k(\mu)_m \mu^m \mu^k) - \gamma( \nabla_k(\mu^k)) \nabla_i(\mu_j) + \gamma R_{rijs} \mu^r \mu^s.
\end{aligned}
\end{equation*}
Moreover, we assume that $\mu$ is harmonic so $\nabla_k(\mu^k) = 0$ and the second to last term vanishes. Now let $s = g^{ij}R_{ij}$, $s' = {g'}^{ij} R'_{ij}$ be the scalar curvatures. Then contracting the above expression for $R'_{ij}$ and using the assumption that $\mu$ is harmonic gives:
\begin{equation*}
\begin{aligned}
s' &= (g^{ij} + \gamma \mu^i \mu^j)( R_{ij} + \gamma \nabla_i( \mu^k) \nabla_j( \mu_k) + \gamma^2 (\nabla_i(\mu_m) \mu^m)(\nabla_j(\mu_n)\mu^n)) \\
& \; \; \; \; \; \; \; \; - (g^{ij} + \gamma \mu^i \mu^j)(\gamma^2 (\nabla_i(\mu)_j) ( \nabla_k(\mu)_m \mu^m \mu^k) ) + (g^{ij} + \gamma \mu^i \mu^j)\gamma R_{rijs} \mu^r \mu^s) \\
&= s + \gamma R_{ij} \mu^i \mu^j + \gamma | \nabla(\mu) |^2 + 2 \gamma^2 | \nabla_{\mu^{\#}}(\mu)|^2 + \gamma R_{rijs} g^{ij} \mu^r \mu^s + \gamma R_{rijs} \mu^i \mu^j \mu^r \mu^s \\
&= s + 2\gamma R_{ij} \mu^i \mu^j + \gamma | \nabla(\mu) |^2 + 2 \gamma^2 | \nabla_{\mu^{\#}}(\mu)|^2.
\end{aligned}
\end{equation*}

\end{proof}

Let $X$ be a compact, oriented, smooth $4$-manifold with $b^+(X) \ge 3$ and $f : X \to X$ an orientation preserving diffeomorphism. Assume further that $b^+(X) \ge 3$. Recall that according to Definition \ref{def:basic} we say that a cohomology class $c \in H^2(X ; \mathbb{Z})$ is {\em $f$-basic} if $c = c_1(\mathfrak{s})$ for a spin$^c$-structure $\mathfrak{s}$ on $X$ which is preserved by $f$, satisfies $d(X , \mathfrak{s}) = -1$ and such that either $SW^{\mathbb{Z}_2}(X , f , \mathfrak{s})$ is non-zero or $SW^{\mathbb{Z}}(X , f , \mathfrak{s})$ is defined and is non-zero.

\begin{proof}[{\bf Proof of Theorem \ref{thm:fadjunction}:}] Let $j : \Sigma \to X$ be an embedded surface of genus $g$ and suppose there exists an isotopy between the embeddings given by $j$ and $f \circ j$. By the isotopy extension theorem \cite[Theorem B]{pal}, there exists a diffeomorphism $g : X \to X$ isotopic to the identity such that $f \circ j = g \circ j$, hence $(g^{-1} \circ f) \circ j = j$. Note that a class $c \in H^2(X ; \mathbb{Z})$ is $f$-basic if and only if it is $g^{-1} \circ f$ basic, since the Seiberg-Witten invariants of a diffeomorphism depend only on the isotopy class. Replacing $f$ by $g^{-1} \circ f$, we may as well assume that $f$ preserves $\Sigma$ (i.e. $f \circ j = j$). Let $N_\Sigma$ denote the normal bundle of $\Sigma$. Then $N_\Sigma$ is an $SO(2)$-bundle  with Euler class $[\Sigma]^2$. Identify a tubular neighbourhood of $\Sigma$ with a neighbourhood of the zero section in $N_\Sigma$. The action of $f$ in such a neighbourhood of $\Sigma$ is then given by the induced action $f_* : N_\Sigma \to N_\Sigma$ of the derivative of $f$ along the normal bundle. Since the identity component of $GL(2,\mathbb{R})$ deformation retracts to $SO(2)$, we can assume after performing an isotopy that $f_*$ is valued in $SO(2)$. In other words, there is a smooth function $\varphi : \Sigma \to SO(2)$ such that the action of $f$ in a suitable tubular neighbourhood of $\Sigma$ has the form
\[
N_\Sigma \to N_\Sigma \quad w \mapsto \varphi( \pi(w)) w
\]
where $\pi : N_\Sigma \to \Sigma$ is the projection. Note that by applying a further isotopy to $f$, we are free to deform $\varphi$ by a smooth homotopy. Let $g_{\Sigma,0}$ be a constant scalar curvature metric on $\Sigma$. Then after applying to a homotopy to $\varphi$ and extending this to a corresponding isotopy of $f$, we can assume that $\varphi : \Sigma \to S^1$ is harmonic with respect to the metric $g_{\Sigma,0}$ (i.e. $\varphi^{-1} d\varphi$ is a harmonic $1$-form on $\Sigma$ with respect to $g_{\Sigma,0}$).

Assume first that $[\Sigma]^2 = 0$ and $g \ge 1$. Let $E_f \to S^1$ be the mapping torus of $f$ which we interpret as a family of $4$-manifolds parametrised by $S^1 = [0,1]/\{0 \sim 1\}$. So $E_f$ is the quotient of $[0,1] \times X$ by the relation $(1,x) \sim (0,f(x))$. A fibrewise metric for $E_f$ may be obtained from a smooth path $\{ g_t \}_{t \in [0,1]}$ of metrics on $X$ such that $g_1 = f^*(g_0)$ and such that the path $g_t$ is constant near $t=0,1$. Since $[\Sigma]^2 = 0$, the normal bundle is trivial and we may write $N_\Sigma = \Sigma \times \mathbb{R}^2$. The action of $f$ in a tubular neighbourhood of $\Sigma$ has the form $f( w , z) = (w , \varphi(w)z)$ for some $\varphi : \Sigma \to SO(2)$. Let $\epsilon_1 > \epsilon_2 > 0$ be chosen so that the annulus $Y_{\epsilon_1,\epsilon_2} = \{ (w,z) | \epsilon_1 > |z| > \epsilon_2 \}$ is contained in the tubular neighbourhood. Then $Y_{\epsilon_1,\epsilon_2}$ is diffeomorphic to the product $(\epsilon_1 , \epsilon_2) \times S^1 \times \Sigma = (\epsilon_1 , \epsilon_2) \times Y$ where $Y = S^1 \times \Sigma$. We define a metric $g_{Y_{\epsilon_1,\epsilon_2}}$ on $Y_{\epsilon_1 , \epsilon_2}$ by taking the product of the standard metric on $(\epsilon_1 , \epsilon_2)$ induced from $\mathbb{R}$ with some metric $g_{Y,0}$ on $Y$. Choose an extension of $g_{Y_{\epsilon_1,\epsilon_2}}$ to a metric on all of $X$ and denote this metric as $g_0$. Let $g_1 = f^*(g_0)$ and for $t \in [0,1]$ define $g_t = tg_1 + (1-t)g_0$. So $\{ g_t \}_{t \in [0,1]}$ defines a family of metrics for the family $E_f$.

Choose some $\epsilon_3 \in (\epsilon_1 , \epsilon_2)$ and identify $Y$ with the subset $\{ (w,z) \; | \; |z| = \epsilon_3 \}$ of $Y_{\epsilon_1,\epsilon_2}$. Then $f$ sends $Y$ to itself and we let $f_Y : Y \to Y$ denote the restriction of $f$ to $Y$. Note that the restriction of $g_0$ to $Y$ is $g_{Y,0}$. More generally, let $g_{Y,t}$ denote the restriction of $g_t$ to $Y$. Note that $g_{Y,t} = t g_{Y,1} + (1-t)g_{Y,0}$, where $g_{Y,1} = f_Y^*(g_{Y,0})$.

Since the normal bundle of $\Sigma$ is trivial, we have that $Y$ separates $X$ into to components $X_+$ and $X_-$ with common boundary $Y$ (the interior of $X_-$ is a tubular neighbourhood of $\Sigma$ and $X_+$ is $X$ minus the interior of $X_-$). Moreover the diffeomorphism $f : X \to X$ respects the decomposition $X = X_+ \cup_Y X_-$. By construction our family of metrics $\{ g_t \}$ when restricted to a collar neighbourhood $(-\epsilon , \epsilon)  \times Y$ of $Y$ is the product of the standard metric on $(-\epsilon , \epsilon)$ induced from $\mathbb{R}$ with the family of metrics $\{ g_{Y,t} \}$ on $Y$.

For any $R>0$ let $X_R$ be the $4$-manifold obtained from $X$ by cutting along $Y$ and inserting a cylinder $[-R , R] \times Y$ with the family of product metrics (namely the standary metric on $[-R,R]$ times the family $\{ g_{Y,t} \}$ of metrics on $Y$). This defines a family $\{ g_t(R) \}_{t \in [0,1]}$ of metrics on $X_R$ for each $R$, moreover $g_1(R) = f_R^*(g_0(R))$, where $f_R : X_R \to X_R$ is the diffeomophism obtained from $f$ by acting as $id \times f_Y$ on the cylinder $[-R,R] \times Y$. Clearly each $X_R$ can be identified with $X$ in such a way that each $f_R$ is isotopic to $f$, hence has the same Seiberg-Witten invariants.

Let $c \in H^2(X ; \mathbb{Z})$ be an $f$-basic class. So there exists a spin$^c$-structure $\mathfrak{s}$ preserved by $f$ with $c = c_1(\mathfrak{s})$ and such that either the mod $2$ Seiberg-Witten invariant of $(X , f , \mathfrak{s})$ is non-zero, or $sgn_+(f)=1$ and the integer Seiberg-Witten invariant of $(X,f,\mathfrak{s})$ is non-zero. In either case this means that the families moduli space $\mathcal{M}(X,f, \mathfrak{s}, \{g_t\} , \{\eta_t\} )$ is non-empty for a generic families perturbation $\{ \eta_t \}_{t \in [0,1]}$. By definition $\eta_t$ is a smooth path of $g_t$-self-dual $2$-forms such that $\eta_1 = f^*(\eta_0)$. Non-emptyness of the moduli space means that for some $t \in [0,1]$ there exists an irreducible solution of the Seiberg-Witten equations for $(X , \mathfrak{s} , g_{t} , \eta_{t})$. We claim this this also implies the existence of a (possibly reducible) solution to the Seiberg-Witten equations for $(X , \mathfrak{s} , g_{t})$ for some $t \in [0,1]$ and with the zero perturbation. To see this, let $H^+_{g_t}$ denote the space of $g_t$-self-dual harmonic $2$-forms and let $w_t \in H^+_{g_t}$ be defined as the $L^2$ orthogonal projection of $c \in H^2(X ; \mathbb{Z})$ to $H^+_{g_t}$, with respect to the metric $g_t$. We note that the $\eta$-perturbed Seiberg-Witten equations in the metric $g_t$ admit a reducible solution if and only if the $L^2$ orthogonal projection of $\eta$ to $H^+_{g_t}$ equals $w_t$. Therefore if $w_t \neq 0$ for all $t \in [0,1]$, then for all families perturbations $\{ \eta_t \}$ in a sufficiently small neighbourhood of the zero, there are no reducibles. Then by taking a sequence of generic families perturbations $\{ \eta_t^{(n)} \}$ converging to zero (in the $\mathcal{C}^\infty$-topology) and using the compactness properties of the Seiberg-Witten equations, we deduce that there exists a solution of the Seiberg-Witten equations for $(X , \mathfrak{s} , g_{t})$ for some $t \in  [0,1]$ with zero perturbation. On the other hand if $w_t = 0$ for some $t_0$, then there exists a (reducible) solution of the Seiberg-Witten equations for $(X , \mathfrak{s} , g_{t_0} )$ with zero perturbation. Hence the claim follows in either case.

The above argument can likewise be applied to the family of metrics $\{ g_t(R) \}_{t \in [0,1]}$ for any $R>0$. Hence for each $R>0$ there exists a $t_R \in [0,1]$ such that $(X_R , f_R , \mathfrak{s} , g_{t_R}(R) )$ admits a solution of the unperturbed Seiberg-Witten equations. But note that the restriction of $g_{t_R}$ to the collar $[-R , R] \times Y$ is the product of the standard metric on $[-R,R]$ with $g_{Y,t_R}$. By compactness of $[0,1]$, we can find a sequence $R_n$ of positive real numbers such that $R_n \to \infty$ as $n \to \infty$ and $t_{R_n} \to t_\infty $ as $n \to \infty$ for some $t_\infty \in [0,1]$. This implies that the sequence of metrics $g_{Y , t_{R_n}}$ converges in the $\mathcal{C}^\infty$-topology to $g_{Y , t_\infty}$.

The argument used in the proof of \cite[Proposition 8]{km} (or alternatively \cite[Section 2.4.2]{nic}) can be adapted to the case where there is a sequence of metrics converging in the $\mathcal{C}^\infty$-topology. It follows that there exists a solution of the Seiberg-Witten equations on $\mathbb{R} \times Y$ with respect to the product ${g}_{\mathbb{R}} \times g_{Y , t_\infty}$ of the standard metric on $\mathbb{R}$ with $g_{Y , t_\infty}$ which is translation invariant in the temporal gauge, and where the spin$^c$-structure is obtained from $\mathfrak{s}$ in the obvious way (restrict $\mathfrak{s}$ to a collar neighbourhood $(-\epsilon , \epsilon) \times Y$ of $Y$ and then extend to $\mathbb{R} \times Y$ by homotopy equivalence).

Recall that $Y = S^1 \times \Sigma$. We now take $g_{Y,0} = g_{S^1} \times g_{\Sigma}$ to be the product of a unit length metric $g_{S^1}$ on $S^1$ with $g_{\Sigma} = \lambda g_{\Sigma,0}$ on $\Sigma$, where $g_{\Sigma,0}$ is a fixed constant scalar curvature metric on $\Sigma$ and $\lambda$ is a positive constant which for the moment is left undetermined. Note that $g_{\Sigma}$ is a constant scalar curvature metric and that $\varphi$ is harmonic with respect to $g_\Sigma$ for any $\lambda$. 

Recall that $g_{Y,t} = t g_{Y,1} + (1-t)g_{Y,0}$ where $g_{Y,1} = f_Y^*(g_{Y,0})$. To proceed further we need to examine the scalar curvature of the metrics $g_{Y,t}$. For this purpose we can view $S^1$ as the quotient of $\mathbb{R}$ with its standard metric by the action of $\mathbb{Z}$ by integer translations. Then we may as well compute the scalar curvature of the pullback metric on $\mathbb{R} \times \Sigma$. Then $f_Y : Y \to Y$ is given by 
\[
f_Y( x , w ) = (x + \varphi(w) , w )
\]
where we view $\varphi : \Sigma \to S^1$ as an $\mathbb{R}/\mathbb{Z}$-valued function such that $\mu = d \varphi$ is a harmonic $1$-form on $\Sigma$ with integral periods. The metric $g_{Y,0}$ lifts to $dx^2 + g_\Sigma$ on $\mathbb{R} \times \Sigma$. It follows that 
\[
g_{Y,1} = f_Y^*(g_{Y,0}) = f_Y^*( dx^2 + g_\Sigma) = (dx + \mu)^2 + g_\Sigma
\]
and hence
\[
g_{Y,t} = g_{Y,0} + t (dx \otimes \mu + \mu \otimes dx + \mu^2).
\]
Let $s$ denote the (constant) scalar curvature of $g_\Sigma$ and $s_t$ the scalar curvature of $g_{Y,t}$. Then from Lemma \ref{lem:scalcurv}, we have
\[
s_t = s +2\gamma \, Ric( \mu , \mu ) +\gamma | \nabla(\mu) |^2 + 2 \gamma^2 | \nabla_{\mu^{\#}}(\mu) |^2
\]
where $Ric$ and $\nabla$ denote the Ricci curvature and Levi-Civita connection of $g_\Sigma$, $\mu^{\#} = (g_\Sigma)^{-1}(\mu)$, $\gamma$ is given by
\[
\gamma = \frac{ t^2-t}{1+ (t-t^2)|\mu|^2}
\]
and where $| \; . \; |$ denotes taking the norm of various tensors induced by $g_\Sigma$. Since $\Sigma$ is $2$-dimensional, we have $Ric = \frac{s}{2} g_\Sigma$ and so
\begin{equation}\label{equ:st1}
\begin{aligned}
s_t &= s +\gamma s |\mu|^2 + \gamma | \nabla(\mu) |^2 + 2 \gamma^2 | \nabla_{\mu^{\#}}(\mu) |^2.
\end{aligned}
\end{equation}
The Gauss-Bonnet theorem gives
\[
s = \frac{8\pi(1-g)}{vol(\Sigma)}
\]
where $vol(\Sigma)$ denotes the volume of $\Sigma$ in the metric $g_\Sigma$. Note that $s \le 0$ as $g \ge 1$. We also see that $\gamma \le 0$ from the definition of $\gamma$ and so $s\gamma \ge 0$. Thus from Equation (\ref{equ:st1}), we obtain
\begin{equation*}
s_t \ge s + \gamma | \nabla(\mu)|^2.
\end{equation*}
Furthermore, one sees that $\gamma \ge -1$ and so
\[
s_t \ge s - | \nabla(\mu)|^2.
\]
Next, we observe that since $g_\Sigma = \lambda g_{\Sigma , 0}$, we have that $vol(\Sigma)$ scales as $\lambda$, $s$ scales as $\lambda^{-1}$ and $|\nabla(\mu)|^2 = g^{ij}g^{kl}\nabla_i(\mu_k)\nabla_j(\mu_l)$ scales as $\lambda^{-2}$ (because $\mu$ and $\nabla$ do not depend on $\lambda$, but $g^{ij}$ scales as $\lambda^{-1}$). Hence, given any $\epsilon > 0$ we can choose a $\lambda$ sufficiently large that $| \nabla(\mu)|^2 vol(\Sigma) < \epsilon$. Hence
\begin{equation}\label{equ:st2}
s_t \ge s - \frac{\epsilon}{vol(\Sigma)}.
\end{equation}
Assume henceforth that such a $\lambda$ has been chosen.

Let $(A , \psi)$ denote a solution of the Seiberg-Witten equations on $\mathbb{R} \times Y$ with respect to the product metric $g_\mathbb{R} \times g_{Y , t_\infty}$ and zero perturbation, which is in the temporal gauge and is translation invariant. Equation (\ref{equ:st2}) gives
\[
-s_t \le s + \frac{\epsilon}{vol(\Sigma)} = \frac{8\pi(1-g) + \epsilon}{vol(\Sigma)}.
\]
Then arguing as in \cite[Section 5]{km}, the Seiberg-Witten equations imply an estimate of the form 
\[
|F_A| \le \frac{4\pi(g-1)+\epsilon/2}{vol(\Sigma)}
\]
and therefore
\[
| \langle [\Sigma] , c \rangle | = \frac{1}{2\pi} \left| \int_\Sigma F_A \right| \le \frac{1}{2\pi} \int_\Sigma |F_A| dvol_g \le (2g-2) + \frac{\epsilon}{4\pi}.
\]
Since $\epsilon$ was arbitrary, this proves Theorem \ref{thm:fadjunction} in the case that $[\Sigma]^2 = 0$ and $g \ge 1$.

Next assume that $[\Sigma]^2 = 0$ and $g=0$. Our argument in this case is modelled on \cite[Lemma 5.1]{fs0}. Proceeding as before we can assume that $f$ preserves $\Sigma$ and that the action of $f$ in a tubular neighbourhood of $\Sigma$ has the form
\[
N_\Sigma \to N_\Sigma \quad w \mapsto \varphi( \pi(w)) w
\]
for some $\varphi : \Sigma \to S^1$. But $g=0$ implies that $H^1(\Sigma ; \mathbb{Z}) = 0$ and hence $\varphi$ is homotopic to the identity. Hence after changing $f$ by a further isotopy, we can assume that $f$ acts as the identity in some tubular neighbourhood $U$ of $\Sigma$. For any integer $n > 0$, let $\Sigma_1 , \dots , \Sigma_n$ denote $n$ parallel copies of $\Sigma$ contained in $U$ and mutually disjoint. This is possible since $[\Sigma]^2 = 0$ means the normal bundle of $\Sigma$ is trivial. Let $X' = X \# \overline{\mathbb{CP}^2}$ denote the connected sum of $X$ with $\overline{\mathbb{CP}^2}$, where the connected sum is performed within $U$. Let $U'$ denote the corresponding open subset of $X'$ obtained from connect summing $\overline{\mathbb{CP}^2}$ to $U$. Then $f$ extends to a diffeomorphism $f' : X' \to X'$ which acts as the identity on $U'$. Let $S \in H^2(X' ; \mathbb{Z})$ represent the exceptional divisor. The blowup formula for families Seiberg-Witten invariants \cite[Theorem 2.2]{liu} implies that $c' = c + S \in H^2(X' ; \mathbb{Z})$ is an $f'$-basic class (note that the blowup formula in \cite{liu} was proved under an orientability assumption which in our case amounts to $sgn_+(f)=1$, however the proof works works equally well for the mod $2$ Seiberg-Witten invariants without any orientability assumption). The class $[\Sigma_n] = n[\Sigma] + S$ can be represented by an embedded sphere $\Sigma_n$ in $U' \subset X'$. Moreover $[\Sigma_n]^2 = -1$, so we can blow down $\Sigma_n$ to get a $4$-manifold $Y_n$. Since the blow down operation takes place in $U'$, we see that $f'$ induces a diffeomorphism $f_n : Y_n \to Y_n$. Writing $c' = c_n + [\Sigma_n]$ with $c_n \in H^2( Y_n ; \mathbb{Z})$, we see by another application of the families blowup formula that $c_n$ is an $f_n$-basic class. Applying the families blowup formula yet again, it follows that $c_n - [\Sigma_n]$ is an $f'$-basic class. But
\[
c_n - [\Sigma_n] = c_n + [\Sigma_n] - 2[\Sigma_n] = c - S - 2 n[\Sigma].
\]
If $[\Sigma]$ is non-torsion, then we have an infinite collection of $f'$-basic classes $\{ c-S-2n[\Sigma] \}$ on $X'$. However, for a given diffeomorphism $f' : X' \to X'$, only finitely many classes can be $f'$-basic. This follows by the same reasoning as in the unparametrised case: we can obtain a bound on solutions of the Seiberg-Witten equations independent of spin$^c$-structure. Therefore $[\Sigma]$ is a torsion class.

We now assume that $[\Sigma]^2 = n > 0$. We will reduce to the case $[\Sigma]^2 = 0$ by a sequence of blowups. This is a well-known technique in the unparametrised case (eg, \cite{km}). By possibly reversing orientation on $\Sigma$, we can assume that $\langle [\Sigma] , c \rangle \ge 0$. As before we can assume that $f$ preserves $\Sigma$ and that the action of $f$ in a tubular neighbourhood of $\Sigma$ takes the form
\[
N_\Sigma \to N_\Sigma \quad w \mapsto \varphi( \pi(w)) w
\]
for some $\varphi : \Sigma \to SO(2)$. Recall also that we are free to change $\varphi$ by an arbitrary smooth homotopy (and change $f$ by a corresponding isotopy). Therefore, we may assume that there exists a non-empty open subset $U \subset \Sigma$ on which $\varphi$ is the identity. It follows that $f$ acts as the identity on $\pi^{-1}(U)$, where $\pi : N_\Sigma \to \Sigma$ is the projection map. In particular, we can find $n$ points $x_1, \dots , x_n \in \Sigma$ such that $f$ acts as the identity in a neighbourhood of each point.

Let $X_n$ be the result of blowing up $X$ at the $n$ points $x_1, \dots , x_n$ of $\Sigma$. Since $f$ acts as the identity in a neighbourhood of these points, we have that $f$ extends naturally to a diffeomorphism $f_n : X_n \to X_n$. Denote by $\mathfrak{s}_i$ a spin$^c$-structure on the $i$-th copy of $\overline{\mathbb{CP}^2}$ in $X_n = X \#^n \overline{\mathbb{CP}^2}$ such that $S_i = c_1(\mathfrak{s}_i)$ represents the exceptional divisor and let $\widetilde{\Sigma}$ denote the proper transform of $\Sigma$ (or more differential geometrically, connect sum $\Sigma$ to $2$-spheres representing $-S_1, \dots , -S_n$ at the points $x_1, \dots , x_n$). Then $\widetilde{\Sigma}$ has the same genus as $\Sigma$,  $[\widetilde{\Sigma}]^2 = 0$ and $f_n$ acts as the identity on a neighbourhood of the $\overline{\mathbb{CP}^2}$ summands. We have that $\widetilde{c} = c + S_1 +  \cdots + S_n$ is $f_n$-basic from the blowup formula for families Seiberg-Witten invariants \cite[Theorem 2.2]{liu}. Therefore, we are in the self-intersection zero case, so if $g \ge 1$, then as shown above we have 
\begin{equation*}
\begin{aligned}
2g-2 &\ge | \langle [\widetilde{\Sigma}] , \widetilde{c} \rangle| \\
& = | \langle [\Sigma] - S_1 - \cdots - S_n , c + S_1 + \cdots + S_n \rangle | \\
&= | \langle [\Sigma] , c \rangle + n | \\
& = |\langle [\Sigma] , c \rangle| + n \quad (\text{since } \langle [\Sigma] , c \rangle \ge 0) \\
&= |\langle [\Sigma] , c \rangle| + [\Sigma]^2,
\end{aligned}
\end{equation*}
which is the adjunction inequality. If $g = 0$, then as above we have shown that $[\widetilde{\Sigma}]$ is a torsion class. But this is impossible since $n > 0$ and each $S_i$ is non-torsion.
\end{proof}

%%%%%%%%%%%%%%%%%%%%%%%%%%%%%%%%%%%%%%%
\section{Non-isotopic embedded surfaces}\label{sec:noniso}

In this section we give an application of Theorem \ref{thm:fadjunction} to construct infinite families of embedded surfaces which are homotopic (in fact continuously isotopic) but not smoothly isotopic.

\begin{theorem}
Let $X$ be one of the $4$-manifolds 
\begin{itemize}
\item[(i)]{$\#^{n} (S^2 \times S^2) \#^n K3$ for $n \ge 2$, or}
\item[(ii)]{$\#^{2n} \mathbb{CP}^2 \#^{m} \overline{\mathbb{CP}^2}$ for $n \ge 2$, $m \ge 10n+1$.}
\end{itemize}
Then there exists a diffeomorphism $f : X \to X$ such that:
\begin{itemize}
\item[(1)]{$f$ is homotopic to the identity.}
\item[(2)]{Let $u \in H^2(X ; \mathbb{Z})$ be any class with $u^2 > 0$. If $\Sigma \to X$ is an embedded surface of genus $g$ representing $u$ and $f^n(\Sigma)$, $f^m(\Sigma)$ are smoothly isotopic for some $n \neq m$, then $2g-2 \ge u^2$.}
\item[(3)]{Let $u \in H^2(X ; \mathbb{Z})$ satisfy $u^2 > 0$. In case (ii) assume moreover that u is a multiple of a primitive ordinary class. Then there exists an embedded surface $\Sigma \to X$ representing $u$ such that the surfaces $\{ f^n(\Sigma) \}_{n \in \mathbb{Z} }$ are all continuously isotopic but mutually non-isotopic smoothly. If $u$ is primitive then $\Sigma$ can be taken to have genus zero. In all other cases we can take $\Sigma$ to have genus $g$ satisfying $2g-2 < u^2$.}
\end{itemize}

\end{theorem}
\begin{proof}
Let $X$ be as in (i) or (ii). Then from Corollary \ref{cor:ctsnotsmooth}, there exists a diffeomorphism $f : X \to X$ which is continuously isotopic for the identity (in particular, $f$ is homotopic to the identity proving (1)) and for which there exists a a spin$^c$-structure $\mathfrak{s}$ preserved by $f$ with $d(X , \mathfrak{s}) = -1$ and $SW^{\mathbb{Z}}(X,f,\mathfrak{s}) \neq 0$. Let $f^n$ denote the $n$-fold composition of $f$, where $n$ is any integer. It follows that $f^n$ preserves $\mathfrak{s}$ and that $SW^{\mathbb{Z}}(X , f^n , \mathfrak{s}) \neq 0$ for any $n \neq 0$. Therefore we can apply Theorem \ref{thm:fadjunction} to $(X,f^n)$ for any $n \neq 0$. Let $u \in H^2(X ; \mathbb{Z})$ be any class with $u^2 > 0$ and suppose $\Sigma \to X$ is an embedded surface of genus $g$ representing $u$. If $f^n(\Sigma)$ and $f^m(\Sigma)$ are smoothly isotopic for some $n \neq m$, then $f^{n-m}(\Sigma)$ is smoothly isotopic to $\Sigma$ and hence Theorem \ref{thm:fadjunction} implies that
\[
2g-2 \ge u^2 + | \langle \Sigma , c \rangle | \ge u^2,
\]
proving (2).

Let $u \in H^2(X ; \mathbb{Z})$ satisfy $u^2 > 0$. In case (ii) assume moreover that $u$ is a multiple of a primitive ordinary class (note that this is automatic in case (i)). We will show that there exists an embedded surface $\Sigma \to X$ of genus $g$ representing $u$ and satisfying $2g-2 < u^2$. In light of the above adjunction inequality, this implies that the surfaces $\{ f^n(\Sigma) \}_{n \in \mathbb{Z}}$ are mutually smoothly non-isotopic. On the other hand, $f$ is continuously isotopic to the identity, so all the surfaces $\{ f^n(\Sigma) \}_{n \in \mathbb{Z}}$ are continuously isotopic.

Let $(\Lambda , \langle \, , \, \rangle)$ be the lattice $\Lambda = H^2(X ; \mathbb{Z})$ equipped with the intersection pairing $\langle \, , \, \rangle$. Let $\Gamma = Aut(\Lambda)$ denote the group automorphisms of $\Lambda$ preserving the intersection form. Any orientation preserving diffeomorphism of $X$ acts on $H^2(X ; \mathbb{Z})$ preserving the intersection form. In all cases we are considering, $X$ can be written in the form $X = (S^2 \times S^2) \# M$ for some smooth simply-connected compact oriented $4$-manifold $M$ with indefinite intersection form. By a theorem of Wall \cite{wall}, this implies that every element of $\Gamma$ can be realised by an orientation preserving diffeomorphism of $X$. If follows that there exists an embedded surface $\Sigma \to X$ of genus $g$ representing $u$ and satisfying $2g-2 < u^2$ if and only if the corresponding statement holds for $\gamma(u)$ for any $\gamma \in \Gamma$. By another theorem of Wall \cite[Theorem 4 and Theorem 6]{wall0}, if $b^+(X), b^-(X) \ge 2$, then $\Gamma$ acts transitively on the set of primitive elements of $\Lambda$ of given norm and type (recall that the type of an element $x \in \Lambda$ refers to whether $x$ is characteristic or ordinary). It is easy to see that the condition $b^+(X), b^-(X) \ge 2$ is satisfied for the $4$-manifolds we are considering.

Since $u^2 > 0$, it follows that $u \neq 0$ and thus we can write $u$ in the form $u = d u_0$, where $d > 0$ is an integer and $u_0$ is primitive. We then have that the $\Gamma$-orbit of $u$ in $\Lambda$ depends only on the values of $d>0$,  $u_0^2 > 0$ and the type of $u_0$.

Now we consider separately cases (i) and (ii). First consider case (i). Note that in this case any primitive class is ordinary, so we need only consider the values of $d$ and $u^2_0$. We again write $X$ as $X = S^2 \times S^2 \# M$ and then $\Lambda = H \oplus \Lambda'$, where $H = H^2(S^2 \times S^2 ; \mathbb{Z})$ and $\Lambda' = H^2(M ; \mathbb{Z})$ are the intersection forms for $S^2 \times S^2$ and $M$. Recall that $H$ has a basis $x,y$ satisfying
\[
\langle x , x \rangle = \langle y , y\rangle = 0, \quad \langle x , y \rangle = 1.
\]
Since $\Lambda$ is even we must have $u_0^2 = 2k$ for some $k > 0$. Then $u'_0 = kx+y$ is primitive and $(u'_0)^2 = 2k = u_0^2$. So there exists some $\gamma \in \Gamma$ such that $u_0 = \gamma( u'_0)$. Setting $u' = d u'_0$ we likewise have $u = \gamma(u')$. From \cite{rub0} it is known that $u' = dkx + dy$ can be represented by an embedded surface $\Sigma \to S^2 \times S^2$ of genus $g$, where $g=0$ if $d=1$ and
\[
g = (dk-1)(d-1) = kd^2 - dk-d+1
\]
if $d>1$. Note that
\[
2g-2 = 2d^2k -2dk-2d < 2d^2k = u^2.
\]
Of course the embedding $\Sigma \to S^2 \times S^2$ representing $u' \in H$ gives rise to an embedding $\Sigma \to X = S^2 \times S^2 \# M$ of the same genus representing $(u' , 0) \in H \oplus \Lambda'$, by taking the connected sum at a point not on $\Sigma$.

Now consider case (ii) and assume $u_0$ is ordinary. Then we can write $X$ as $X = S^2 \times S^2 \# \mathbb{CP}^2 \# M$ for $M = \#^{2n-2} \mathbb{CP}^{2} \#^{m-1} \overline{\mathbb{CP}^2}$ and $\Lambda = K \oplus \Lambda'$, where $K = H^2( S^2 \times S^2 \# \mathbb{CP}^2 ; \mathbb{Z})$ and $\Lambda' = H^2(M ; \mathbb{Z})$ are the intersection forms for $S^2 \times S^2 \# \mathbb{CP}^2$ and $M$. Note that $K$ has a basis $x,y,z$ satisfying
\[
\langle x , x \rangle = \langle y , y\rangle = \langle x , z \rangle = \langle y , z \rangle = 0, \quad \langle x , y \rangle = \langle z , z \rangle = 1.
\]
Consider separately the cases where $u_0^2$ is even or odd. In the even case $u_0^2 = 2k$ and we again consider $u'_0 = kx+y$. It is easy to see that $u'_0$ is primitive, ordinary and $(u'_0)^2 = 2k$, so there exists $\gamma \in \Gamma$ for which $u_0 = \gamma( u'_0)$. Proceeding exactly as in case (i), we deduce that $u$ can be represented by an embedded surface $\Sigma \to X$ of genus $g$, where $2g-2 < u^2$. In the odd case $u_0^2 = 2k+1$, $k \ge 0$ and we consider $u'_0 = kx+y+z$. Then $u'_0$ is primitive, ordinary and $(u'_0)^2 = 2k+1$. So there exists $\gamma \in \Gamma$ such that $u_0 = \gamma(u'_0)$. Hence $u = \gamma(u')$, where $u' = du'_0$. Arguing as in case (i), the class $dkx+dy$ can be represented by an embedded surface $\Sigma \to S^2 \times S^2$ of genus $g$, where $g=0$ if $d=1$ or $k=0$ and $g = (dk-1)(d-1)$ otherwise. The class $dz$ can of course be represented by an embedded surface in $\mathbb{CP}^2$ of genus $(d-1)(d-2)/2$. Taking a connected sum, it follows that the class $u' \in \Lambda$ can be represented by an embedded surface $\Sigma \to X$ of genus $g$, where $g = (d-1)(d-2)/2$ if $k=0$ and $(dk-1)(d-1)+(d-1)(d-2)/2$ otherwise. In either case one easily finds that $2g-2 < (2k+1)d^2 = u^2$.

Lastly note that in the above constructions, if $u$ is primitive then we have found a representative for $u$ of genus $0$. This completes the proof of (3).

\end{proof}

%%%%%%%%%%%%%%%%%%%%%%%%%%%%%%%%%%%%%%%%%%%%%%%%%%%%%%%%%%%%%%%%%%%%%%%%%%%%%%%%
%%%%%%%%%%%%%%%%%%%%%%%%%%%%%%%%%%%%%%%%%%%%%%%%%%%%%%%%%%%%%%%%%%%%%%%%%%%%%%%%
%%%%%%%%%%%%%%%%%%%%%%%%%%%%%%%%%%%%%%%%%%%%%%%%%%%%%%%%%%%%%%%%%%%%%%%%%%%%%%%%
%%%%%%%%%%%%%%%%%%%%%%%%%%%%%%%%%%%%%%%%%%%%%%%%%%%%%%%%%%%%%%%%%%%%%%%%%%%%%%%%

\bibliographystyle{amsplain}

\begin{thebibliography}{99}
\bibitem{akmr}D. Auckly, H. J. Kim, P. Melvin, D. Ruberman, Stable isotopy in four dimensions. {\em J. Lond. Math. Soc.} (2) {\bf 91} (2015), no. 2, 439-463. 
\bibitem{akmrs}D. Auckly, H. J. Kim, P. Melvin, D. Ruberman, H. Schwartz, Isotopy of surfaces in $4$-manifolds after a single stabilization. {\em Adv. Math.} {\bf 341} (2019), 609-615. 
\bibitem{ak}S. Akbulut, {\em Isotoping $2$-spheres in $4$-manifolds}. Proceedings of the G\"okova Geometry-Topology Conference 2014, 264-266, G\"okova Geometry/Topology Conference (GGT), G\"okova, 2015.
\bibitem{bk}D. Baraglia, H. Konno, A gluing formula for families Seiberg--Witten invariants. {\em Geom. Topol.} {\bf 24} (2020), no. 3, 1381-1456. 
\bibitem{cer}J. Cerf, Sur les diff\'eomorphismes de la sph\`ere de dimension trois $(\Gamma_4=0)$. Lecture Notes in Mathematics, No. 53 Springer-Verlag, Berlin-New York 1968 xii+133 pp. 
\bibitem{fs0}R. Fintushel, R. Stern, Immersed spheres in $4$-manifolds and the immersed Thom conjecture. {\em Turkish J. Math.} {\bf 19} (1995), no. 2, 145-157. 
\bibitem{fs}R. Fintushel, R. Stern, Surfaces in $4$-manifolds. {\em Math. Res. Lett.} {\bf 4} (1997), no. 6, 907-914. 
\bibitem{gab}D. Gabai, The 4-dimensional light bulb theorem, arXiv:1705.09989 (2017). 
\bibitem{gom}R. E. Gompf, Sums of elliptic surfaces. {\em J. Differential Geom.} {\bf 34} (1991), no. 1, 93-114. 
\bibitem{km}P. B. Kronheimer, T. S. Mrowka, The genus of embedded surfaces in the projective plane. {\em Math. Res. Lett.} {\bf 1} (1994), no. 6, 797-808. 
\bibitem{law}T. Lawson, The minimal genus problem. {\em Exposition. Math.} {\bf 15} (1997), no. 5, 385-431. 
\bibitem{liu}A.-K. Liu, Family blowup formula, admissible graphs and the enumeration of singular curves. I. {\em J. Differential Geom.} {\bf 56} (2000), no. 3, 381-579. 
\bibitem{mst}J. W. Morgan, Z. Szab\'o, C. H. Taubes, A product formula for the Seiberg-Witten invariants and the generalized Thom conjecture. {\em J. Differential Geom.} {\bf 44} (1996), no. 4, 706-788. 
\bibitem{nic} L. I. Nicolaescu, {\em Notes on Seiberg-Witten theory}. Graduate Studies in Mathematics, 28. American Mathematical Society, Providence, RI, 2000. xviii+484 pp.
\bibitem{os}P. Ozsv\'ath, Z. Szab\'o, The symplectic Thom conjecture. {\em Ann. of Math.} (2) {\bf 151} (2000), no. 1, 93-124. 
\bibitem{pal}R. S. Palais, Local triviality of the restriction map for embeddings. {\em Comment. Math. Helv.} {\bf 34} (1960), 305-312. 
\bibitem{qui}F. Quinn, Isotopy of $4$-manifolds. {\em J. Differential Geom.} {\bf 24} (1986), no. 3, 343-372.
\bibitem{rub0}D. Ruberman, The minimal genus of an embedded surface of non-negative square in a rational surface. {\em Turkish J. Math.} {\bf 20} (1996), no. 1, 129-133. 
\bibitem{rub1}D. Ruberman, An obstruction to smooth isotopy in dimension $4$. {\em Math. Res. Lett.} {\bf 5} (1998), no. 6, 743-758. 
\bibitem{rub2}D. Ruberman, A polynomial invariant of diffeomorphisms of $4$-manifolds. Proceedings of the Kirbyfest (Berkeley, CA, 1998), 473-488, Geom. Topol., Coventry, 1999. 
\bibitem{rub3}D. Ruberman, Positive scalar curvature, diffeomorphisms and the Seiberg-Witten invariants. {\em Geom. Topol.} {\bf 5} (2001), 895-924. 
\bibitem{wall0}C. T. C. Wall, On the orthogonal groups of unimodular quadratic forms. {\em Math. Ann.} {\bf 147} (1962), 328-338. 
\bibitem{wall}C. T. C. Wall, Diffeomorphisms of $4$-manifolds, {\em J. London Math. Soc.} {\bf 39} (1964) 131-140. 

\end{thebibliography}

\end{document}